\documentclass[a4paper,12pt]{amsart}  
\usepackage{a4wide,amsmath,amssymb,amsthm}     
\usepackage[utf8]{inputenc}
\usepackage{enumitem}
\usepackage{color}

\theoremstyle{plain}      
 
\newtheorem{thm}{Theorem}[section]

\newtheorem{lemma}[thm]{Lemma}

\theoremstyle{definition}      
\newtheorem{defn}[thm]{Definition}     
\newtheorem{definition}[thm]{Definition}

\DeclareMathAlphabet{\doba}{U}{msb}{m}{n}

\def\T{\mathrm{T}}

\def\Ric{{\mathrm{Ric}}}

\def\di{{\rm d}}

\let\<\langle 
\let\>\rangle

\newcommand{\definedas}{\mathrel{\raise.095ex\hbox{\rm :}\mkern-5.2mu=}}

\title{}

\author{Andrei Moroianu, Mihaela Pilca}

\address{Andrei Moroianu \\ Université Paris-Saclay, CNRS,  Laboratoire de mathématiques d'Orsay, 91405, Orsay, France, and Institute of Mathematics “Simion Stoilow” of the Romanian Academy, 21 Calea Grivitei, 010702
Bucharest, Romania}
\email{andrei.moroianu@math.cnrs.fr}

\address{Mihaela Pilca\\Fakult\"at f\"ur Mathematik\\
Universit\"at Regensburg\\Universit\"atsstr. 31 
D-93040 Regensburg, Germany}
\email{mihaela.pilca@mathematik.uni-regensburg.de}



\begin{document}   
	
	\title{Einstein metrics on conformal products}

	\begin{abstract}
		We show that under some natural geometric assumption, Einstein metrics on conformal products of two compact conformal manifolds are warped product metrics.
	\end{abstract}
	\maketitle

	\section{Introduction}
	
	Given two Riemannian manifolds $(M_1,g_1)$, $(M_2,g_2)$, the product $M_1\times M_2$ carries a na\-tural metric $g=g_1+g_2$ (the Riemannian product metric) whose Levi-Civita connection has reducible holonomy. Conversely, the local de Rham theorem states that every Riemannian manifold $(M,g)$ whose Levi-Civita connection has reducible holonomy, is locally isometric to a Riemannian product.
	
	However, things are more complicated in the category of conformal manifolds. The notion of product is no longer canonically defined, and there is no distinguished connection playing the role of the Levi-Civita connection as in Riemannian geometry. Recall that a conformal class on a smooth manifold $M$ is an equivalence class $c$ of Riemannian metrics for the equivalence relation defined by
	$$g\sim g'\iff\exists f\in C^\infty(M),\ g'=e^{2f}g.$$ 
	
	By the very definition, every Riemannian metric $g$ determines a conformal class denoted $[g]$. A linear connection $\nabla$ on $(M,c)$ is called conformal if $\nabla g=-2\theta^g\otimes g$ for every Riemannian metric $g\in c$, where $\theta^g$ is a 1-form called the Lee form of $\nabla$ with respect to $g$. This $1$-form transforms according to the rule $\theta^{g'}=\theta^g-df$ for every other conformally equivalent metric $g'=e^{2f}g$.
	
	In the conformal setting one has to replace the Levi-Civita connection by the set of {\em torsion-free} conformal connections, called Weyl connections. This set is an affine space modeled on the vector space of real 1-forms. A Weyl connection is called closed (resp. exact) if its Lee form with respect to each metric in the conformal class is closed (resp. exact). From the above transformation rule it readily follows that a Weyl connection is closed (resp. exact) if and only if it is locally (resp. globally) the Levi-Civita connection of a metric in the given conformal class.
	
	Unlike the Riemannian case, given two conformal manifolds $(M_1,c_1)$, $(M_2,c_2)$, the product $M_1\times M_2$ is no longer endowed with a natural ``product'' conformal structure, but rather with a set of conformal structures obtained by choosing Riemannian metrics $g_i\in c_i$ and functions $f_i\in C^\infty(M_1\times M_2)$ for $i\in\{1,2\}$ and defining $c=[e^{2f_1}g_1+e^{2f_2}g_2]$. These conformal classes are called conformal product structures. Each conformal product structure carries a unique compatible Weyl connection whose holonomy preserves the decomposition $\T M=\T M_1\oplus \T M_2$ and conversely, every conformal class carrying a Weyl connection with reducible holonomy is locally a conformal product structure \cite[Thm. 4.3]{bm2011}. 
	
	The aim of this paper is to study conformal product structures on compact manifolds $M=M_1\times M_2$ containing an Einstein metric. Since every conformal class on a compact surface contains Einstein metrics, we will implicitly assume throughout the text that dim$(M)\ge 3$. It turns out that this problem can be understood as part of a long term classification project for compact Riemannian manifolds $(M,g)$ with special holonomy, carrying a Weyl connection $\nabla$ (different from the Levi-Civita connection of $g$), which also has special holonomy. 
	
	Some parts of this project have already been carried out recently. Indeed, when $\nabla$ is exact, this reduces to the study of conformal classes containing two non-homothetic Riemannian metrics with special holonomy, and was solved in \cite{mmp2020} and \cite{m2019}. The case where $\nabla$ is closed but non-exact, was solved in \cite{bfm2023}. It is thus natural to consider the remaining case, where $\nabla$ is non-closed.
	
	Using the Merkulov-Schwachhöfer classification \cite{ms1999} of holonomy groups of torsion-free connections applied to the special case of Weyl structures, we obtain that if the dimension of $M$ is different from $4$, then a non-closed Weyl connection $\nabla$ has special holonomy if and only if it is reducible (whence locally the adapted Weyl connection of a conformal product). Moreover, according to the Berger-Simons holonomy theorem, if the Levi-Civita connection of $g$ has special holonomy, then $g$ is either reducible, or Kähler, or Einstein. This last case is thus contained in the problem mentioned above (but of course the inclusion is strict, since not every Einstein metric has special holonomy). 
	
	It turns out that this problem is too hard in full generality. In order to attack it, we make a simplification, namely we assume that the restriction  to one of the factors of the conformal product of the Lee form of the reducible Weyl structure $\nabla$ with respect to $g$, is $\nabla$-parallel in the direction of the second factor. Equivalently, we are looking for Einstein metrics of the form $e^{2f_1}g_1+e^{2f_2}g_2$ on $M_1\times M_2$, where $f_1$ only depends on $M_2$ and $f_2$ is any function on $M_1\times M_2$. Such metrics generalize the so called doubly warped metrics, which have the same expression $e^{2f_1}g_1+e^{2f_2}g_2$, except that $f_1$ is a function on $M_2$ and $f_2$ is a function on $M_1$. Our main result is the following:
	\begin{thm} \label{mainthm}
		Let $(M_1,c_1)$ and $(M_2,c_2)$ by two compact conformal manifolds such that $\mathrm{dim}(M_1\times M_2)\geq 3$ and let $c$ be a conformal product structure on $M_1\times M_2$, with adapted Weyl connection $\nabla$. Assume that $c$ contains an Einstein metric $g$, such that the restriction to $\T M_2$ of the Lee form of $\nabla$ with respect to $g$ is $\nabla$-parallel in the direction of $\T M_1$. Then there exist metrics $h_i\in c_i$ such that $c=[h_1+h_2]$, i.e. the Einstein metric $g$ is conformal to a product metric. 
	\end{thm}
	
By  \cite[Thm. 3.2 and Cor. 3.4]{kr2016}, this can only happen if $g$ is a warped product metric. A complete classification of warped product Einstein metrics on compact manifolds is not yet available, except when the base of the warped product is 1-dimensional.
	
	{\bf Acknowledgments.} This work was supported by the Procope Project No. 57650868 (Germany) /  48959TL (France).

	\section{Preliminaries}
	
	\subsection{Weyl connections}
		A {\it Weyl connection} on a conformal manifold $(M,c)$ is a torsion-free linear connection $\nabla$ which preserves the conformal class $c$ in the sense that for each metric $g\in c$, there exists a unique $1$-form $\theta^g\in\Omega^1(M)$, called the {\it Lee form} of $D$ with respect to $g$, such that 
	\begin{equation}\label{dg}\nabla g=-2\theta^g\otimes g.
	\end{equation}
	
	The Weyl connection $\nabla$ is then related to the Levi-Civita covariant derivative $\nabla^g$ by the well known formula 
	\begin{equation}\label{weylstr}
		\nabla_X=\nabla^g_X+\theta^g(X)\mathrm{Id} + \theta^g\wedge X, \quad \forall X\in  \mathrm{T}M,
	\end{equation}
	where $\theta^g\wedge X$ is the skew-symmetric endomorphism of $\mathrm{T}M$ defined by $$(\theta^g\wedge X)(Y):=\theta^g(Y) X-g(X,Y)(\theta^g)^\sharp.$$
	
	A Weyl connection $D$ is called {\it closed} if it is locally the Levi-Civita connection of a (local) metric in $c$ and is called {\it exact} if it is the Levi-Civita connection of a globally defined metric in~$c$. Equivalently, $D$ is closed (resp. exact) if its Lee form with respect to one (and hence to any) metric in $c$ is  closed (resp. exact).	

	\subsection{Differential operators on products} Let $M=M_1\times M_2$ be a product manifold and let $\pi_i:M\to M_i$ denote the standard projections for $i=1,2$. We denote by $n_1,n_2$ the dimensions of $M_1,M_2$ and by $n:=n_1+n_2$ the dimension of $M$.  
	For each $0\le k\le n$, the  bundle of $k$-forms of $M$ splits into direct sums 
	$$\Lambda^kM=\bigoplus_{p+q=k}\pi_1^*(\Lambda^pM_1)\otimes \pi_2^*(\Lambda^qM_2)=:\Lambda^{p,q}M,$$
	where of course the notation $\Lambda^{p,q}M$ is specific to this product setting, and should not be confused with the Dolbeault decomposition on complex manifolds.
The exterior differential on $M$ maps $C^\infty(\Lambda^{p,q}M)$ onto  $C^\infty(\Lambda^{p+1,q}M\oplus \Lambda^{p,q+1}M)$. The projections on the two factors of this direct sum are first order differential operators denoted by $\di_1$ and $\di_2$ which satisfy the relations:
\begin{equation}\label{dd}\di=\di_1+\di_2,\qquad \di_1^2=\di_2^2=\di_1\di_2+\di_2\di_1=0.\end{equation}
Assume now that $g_1$, $g_2$ are Riemannian metrics on $M_1$, $M_2$. The formal adjoints of $\di_i$ with respect to the Riemannian metric $g_1+g_2$ on $M$ are denoted by $\delta_i$. If $\{e_\alpha\}_{1\le \alpha\le n_1}$ denotes a local orthonormal basis of $\T M_1$, inducing a local frame of the distribution $\pi_1^*(\T M_1)\subset \T M$, then 
\begin{equation}\label{ddelta}
\di_1=\sum_{\alpha=1}^{n_1}e_\alpha\wedge\nabla^{g_1+g_2}_{e_\alpha},\qquad \delta_1=-\sum_{\alpha=1}^{n_1}e_\alpha\lrcorner\nabla^{g_1+g_2}_{e_\alpha}.
\end{equation}
These operators are clearly conjugate with the corresponding operators on the factors, in the sense that if $\omega$ is an exterior form on $M_1$, then 
$$\di_1(\pi_1^*\omega)=\pi_1^*(\di^{M_1} \omega),\qquad \delta_1(\pi_1^*\omega)=\pi_1^*(\delta^{g_1} \omega),\qquad (\di_1\delta_1+\delta_1\di_1)(\pi_1^*\omega)=\pi_1^*(\Delta^{g_1}\omega).$$
We denote by $\Delta_1:=\di_1\delta_1+\delta_1\di_1$.

\begin{lemma}\label{lemmasum}
	If a function $f\in\mathcal{C}^\infty(M)$ satisfies $\di_1\di_2 f=0$, then $f$ is a sum of two functions, each of them depending only on one of the factors, \emph{i.e.} there exist functions $a_i\in\mathcal{C}^{\infty}(M_i)$ such that $f=a_1+a_2$.
\end{lemma}	
\begin{proof} Let $X_1\in \mathcal{C}^\infty(M_1)$ be any vector field. Clearly $X_1\lrcorner (\di_2 f)=0$ and by the Cartan formula together with \eqref{dd} we can write
$$\mathcal{L}_{X_1}(\di_2f)=\di(X_1\lrcorner (\di_2 f))+X_1\lrcorner(\di \di_2 f)=X_1\lrcorner(\di_1 \di_2 f)=0.$$
This shows that the 1-form $\di_2 f$ is basic with respect to the projection $\pi_2:M_1\times M_2\to M_2$, so there exists a 1-form $\omega_2\in \Omega^1(M_2)$ such that $\di_2f=\pi_2^*\omega_2$. 

For every $x_1\in M_1$ we denote by $i_{x_1}$ the inclusion $M_2\to M_1\times M_2$, given by $x_2\mapsto (x_1,x_2)$. Then $\pi_2\circ i_{x_1}=\mathrm{id}_{M_2}$, whence $\omega_2=i_{x_1}^*\pi_2^*\omega_2=i_{x_1}^*\di_2f=\di^{M_2}(i_{x_1}^*f)$. This shows that for every $x_1\in M_1$, the function $i_{x_1}^*f\in \mathcal{C}^\infty(M_2)$ is a primitive of $\omega_2$. In particular $\omega_2=\di^{M_2}a_2$ is exact on $M_2$, and by connectedness, for every $x_1\in M_1$, there exists a constant $a_1(x_1)$ such that $i_{x_1}^*f=a_1(x_1)+a_2$. In other words $f(x_1,x_2)=a_1(x_1)+a_2(x_2)$ for every $(x_1,x_2)\in M$, and the function $a_1$ is smooth on $M_1$ since $f$ is smooth on $M$.
\end{proof}

	\subsection{Conformal product structures}
	\begin{defn} Consider two conformal manifolds $(M_1,c_1)$ and $(M_2,c_2)$.  A conformal product structure on $M:=M_1\times M_2$ is a conformal class $c$ such that the two canonical projections $\pi_i:(M,c)\to(M_i,c_i)$ are orthogonal conformal submersions. 
	\end{defn}
	Equivalently, for every $x:=(x_1,x_2)\in M$ and for every Riemannian metrics $g_i\in c_i$ and $g\in c$, there exist real numbers $f_1(x),f_2(x)$ such that for all tangent vectors $X_i\in\T_xM_i\subset \T_xM$ one has $g(X_1,X_2)=0$ and $g(X_i,X_i)=e^{2f_i(x)}g_i(X_i,X_i)$. Consequently, for every choice of Riemannian metrics $g_1$ and $g_2$ in the conformal classes $c_1$ and $c_2$, every conformal product structure can be defined by two functions $f_1$ and $f_2$ on $M$ by the formula $c:=[g]$, where $g:=e^{2f_1}g_1+e^{2f_2}g_2$. Clearly, the conformal class $c$ only depends on the difference $f_1-f_2$. This motivates the following definition:
	\begin{definition}
		Let $M_1$ and $M_2$ be two manifolds. A Riemannian metric $g$ on $M_1\times M_2$ of the form $g=e^{2f_1}g_1+e^{2f_2}g_2$, where $f_1$ and $f_2$ are functions on $M_1\times M_2$ and $g_1$, $g_2$ are Riemannian metrics on $M_1$, resp. $M_2$, is called a \emph{conformal product metric.}
	\end{definition}

For $i\in\{1,2\}$, the vector fields on $M_i$ will be denoted by the index $i$, \emph{e.g.} $X_i, Y_i, Z_i$. Each vector field $X_i$ on $M_i$ naturally induces a vector field on $M$, denoted by $\widetilde{X_i}$, so we have the inclusion $\mathcal{C}^{\infty}(\mathrm{T}M_i)\subset \mathcal{C}^{\infty}(\mathrm{T}(M_1\times M_2))$. Let us remark that the Lie bracket between vector fields arising from the different factors $M_1$ and $M_2$ vanishes, \emph{e.g.}  $[\widetilde{X_1},\widetilde{X_2}]=0$. In order to keep the notation as simple as possible, we will identify from now on each vector field $X_i$ on $M_i$ with the corresponding vector field $\widetilde{X_i}$ on $M$ and in the sequel it will be clear from the context whether the vector field is considered on $M$ or on one of its factors.

\begin{lemma}\label{lemmaLC}
	The Levi-Civita connection $\nabla^g$ of a conformal product metric $g=e^{2f_1}g_1+e^{2f_2}g_2$ is given by the following formulas, for all vector fields $X_i, Y_i\in\mathcal{C}^{\infty}(\mathrm{T}M_i)$, $i\in\{1,2\}$:
\begin{equation}\label{connLC1}
	\nabla^g_{X_1}Y_1=\nabla^{g_1}_{X_1}Y_1+X_1(f_1)Y_1+Y_1(f_1)X_1-g(X_1, Y_1)\di f_1^{\#_g},
\end{equation}	
\begin{equation}\label{connLC1bis}
	\nabla^g_{X_2}Y_2=\nabla^{g_2}_{X_2}Y_2+X_2(f_2)Y_2+Y_2(f_2)X_2-g(X_2, Y_2)\di f_2^{\#_g},
\end{equation}	
\begin{equation}\label{connLC2}
	\nabla^g_{X_1}X_2=\nabla^g_{X_2}X_1=X_1(f_2)X_2+X_2(f_1)X_1.
\end{equation}	
In particular, from \eqref{connLC1} it follows that
\begin{equation}\label{connLC11}
	g(\nabla^g_{X_1}Y_1, X_2)=-g(X_1, Y_1)X_2(f_1).
\end{equation}	
\end{lemma}

\begin{proof}
We consider vector fields  $X_1, Y_1, Z_1\in\mathcal{C}^{\infty}(\mathrm{T}M_1)$ and \mbox{$X_2, Y_2, Z_2\in\mathcal{C}^{\infty}(\mathrm{T}M_2)$} and we compute using the Koszul formula:
\begin{equation*}
	2g(\nabla^g_{X_1} Y_1, Z_2)=-Z_2(g(X_1, Y_1))=-Z_2(e^{2f_1}g_1(X_1, Y_1))=-2\di f_1(Z_2)g(X_1, Y_1),
\end{equation*}	
\begin{equation*}
\begin{split}	
	2g(\nabla^g_{X_1} Y_1, Z_1)=&X_1(g(Y_1, Z_1))+Y_1(g(X_1, Z_1))-Z_1(g(X_1, Y_1))\\
	&+g([X_1, Y_1], Z_1)-g([X_1, Z_1], Y_1)-g([Y_1, Z_1], X_1)\\
	=&2 g(\nabla^{g_1}_{X_1} Y_1, Z_1)+2X_1(f_1)g(Y_1, Z_1)+2Y_1(f_1)g(X_1, Z_1)-2Z_1(f_1)g(X_1, Y_1),
\end{split}	
\end{equation*}	
which together yield \eqref{connLC1}. The Equation \eqref{connLC1bis} follows then by symmetry, permuting the indexes in \eqref{connLC1}. We further compute using the Koszul formula:
\begin{equation*}
		2g(\nabla^g_{X_1} X_2, Y_1)=X_2(g(X_1, Y_1))=X_2(e^{2f_1}g(X_1, Y_1))=2X_2(f_1)g(X_1, Y_1),
\end{equation*}	
\begin{equation*}
	2g(\nabla^g_{X_1} X_2, Y_2)=X_1(g(X_2, Y_2))=X_1(e^{2f_2}g(X_2, Y_2))=2X_1(f_2)g(X_2, Y_2),
\end{equation*}	
which together yield \eqref{connLC2}.
\end{proof}

	For every conformal product structure $c=[e^{2f_1}g_1+e^{2f_2}g_2]$ on $M$, there exists a unique Weyl connection $\nabla$ whose holonomy preserves the decomposition $\T M=\T M_1\oplus \T M_2$. This connection is called  {\em adapted} and its Lee form with respect to $g:=e^{2f_1}g_1+e^{2f_2}g_2$ reads $\theta^g=-\di_1f_2-\di_2 f_1$ (cf. \cite[\S 6.1]{bm2011}). Note that the adapted Weyl connection is closed if and only if $\di_1\di_2(f_1-f_2)=0$.
	
Let us remark that the vector fields tangent to one of the two factors $M_1$ or $M_2$ are parallel with respect to the adapted Weyl connection in the direction of the other factor, namely for all $X_i\in\mathcal{C}^{\infty}(\mathrm{T}M_i)$ we have:
\begin{equation}\label{nabla12}
	\begin{split}
		\nabla_{X_1} X_2&\overset{\eqref{weylstr}}{=}\nabla^g_{X_1} X_2+\theta^g(X_1)X_2+\theta^g(X_2)X_1\\
		&\overset{\eqref{connLC2}}{=}X_1(f_2)X_2+X_2(f_1)X_1-X_1(f_2)X_2-X_2(f_1)X_1=0.
	\end{split}	
\end{equation}
	
	\subsection{Curvature of conformal product metrics}
	The purpose of this section is to establish the formulas for the Riemannian curvature tensor and the Ricci curvature of a conformal product metric, which generalize the well-known O'Neill formulas for warped products. Let $g$ be such a metric given as $g=e^{2f_1}g_1+e^{2f_2}g_2$ on \mbox{$M$}.
	We start by computing the Riemannian curvature tensor of $g$. 
	
	\begin{lemma}\label{lemmariemtensor}
		The Riemannian curvature tensor $R^g$ of the metric $g=e^{2f_1}g_1+e^{2f_2}g_2$ on $M$ is given by the following formulas, for all vector fields $X_1, Y_1, Z_1\in\mathcal{C}^{\infty}(\mathrm{T}M_1)$ which are $\nabla^{g_1}$-parallel at the point where the computation is done and $X_2, Y_2, Z_2\in\mathcal{C}^{\infty}(\mathrm{T}M_2)$, which are $\nabla^{g_2}$-parallel at the same point:
		\begin{equation}\label{Riem1112}
			\begin{split}
			R^g(X_1, Y_1, Z_1, X_2)=&g(X_1, Z_1)[Y_1(X_2(f_1))-Y_1(f_2)\cdot X_2(f_1)]\\
			&-g(Y_1, Z_1)[X_1(X_2(f_1))-X_1(f_2)\cdot X_2(f_1)],
		\end{split}
		\end{equation}	
	\smallskip
	\begin{equation}\label{Riem1111}
		\begin{split}
			R^g(X_1, Y_1, Y_1, X_1)=&R^{g_1}(X_1, Y_1, Y_1, X_1)+2X_1(Y_1(f_1))g(X_1,Y_1)+(g(X_1, Y_1))^2|\di f_1|^2_g\\
			&-2X_1(f_1)\cdot Y_1(f_1)\,g(X_1, Y_1)
			- X_1(X_1(f_1))\, |Y_1|^2_g- Y_1(Y_1(f_1))\, |X_1|^2_g\\
			&+(X_1(f_1))^2|Y_1|^2_g+(Y_1(f_1))^2|X_1|^2_g-|\di f_1|^2_g|X_1|^2_g|Y_1|^2_g,
				\end{split}
	\end{equation}	
	\smallskip
\begin{equation}\label{Riem1221}
\begin{split}
	R^g(X_1, X_2, X_2, X_1)=&- (X_1(f_2))^2\, |X_2|_g^2-X_1(X_1(f_2))\, |X_2|_g^2+2 X_1(f_1)\cdot X_1(f_2)|X_2|_g^2\\
	&-X_2(X_2(f_1))\, |X_1|_g^2+2X_2(f_1)\cdot X_2(f_2)\, |X_1|_g^2-(X_2(f_1))^2\, |X_1|_g^2\\
	&-g(\di f_1, \di f_2)\, |X_1|_g^2|X_2|^2_g.
\end{split}
\end{equation}	

By symmetry, permuting the indexes, we also obtain the analogous formulas to \eqref{Riem1112} and~\eqref{Riem1111}:

\begin{equation}\label{Riem2221}
	\begin{split}
	R^g(X_2, Y_2, Z_2, X_1)=&g(X_2, Z_2)[Y_2(X_1(f_2))-Y_2(f_1)\cdot X_1(f_2)]\\
	&-g(Y_2, Z_2)[X_2(X_1(f_2))-X_2(f_1)\cdot X_1(f_2)],
\end{split}
\end{equation}	
\smallskip
\begin{equation}\label{Riem2222}
	\begin{split}
		R^g(X_2, Y_2, Y_2, X_2)=&R^{g_2}(X_2, Y_2, Y_2, X_2)+2X_2(Y_2(f_2))g(X_2,Y_2)+(g(X_2, Y_2))^2|\di f_2|^2_g\\
		&-2X_2(f_2)\cdot Y_2(f_2)\,g(X_2, Y_2)
		- X_2(X_2(f_2))\, |Y_2|^2_g- Y_2(Y_2(f_2))\, |X_2|^2_g\\
		&+(X_2(f_2))^2|Y_2|^2_g+(Y_2(f_2))^2|X_2|^2_g-|\di f_2|^2_g|X_2|^2_g|Y_2|^2_g.
	\end{split}
\end{equation}	
	\end{lemma}
	
\begin{proof}	
	Since $X_1$ and $Y_1$ are $\nabla^{g_1}$-parallel at the point where the computation is done, we have by the definition of the Riemannian curvature tensor:
	\begin{equation}\label{defriemcurv}
		R^g(X_1, Y_1, Z_1, X_2)=-R^g(X_1, Y_1, X_2, Z_1)=-g(\nabla^g_{X_1}\nabla^g_{Y_1} X_2, Z_1)+g(\nabla^g_{Y_1}\nabla^g_{X_1} X_2, Z_1).
	\end{equation}	
	The first term on the right-hand side is then computed by applying the formulas obtained for the Levi-Civita connection in Lemma \ref{lemmaLC}:
	\begin{equation*}
		\begin{split}
			g(\nabla^g_{X_1}\nabla^g_{Y_1} X_2, Z_1)&\overset{\eqref{connLC2}}{=}g(\nabla^g_{X_1}(Y_1(f_2)X_2+X_2(f_1)Y_1), Z_1)\\
			&=Y_1(f_2)g(\nabla^g_{X_1}X_2, Z_1)+X_1(X_2(f_1))g(Y_1, Z_1)+X_2(f_1)g(\nabla^g_{X_1}Y_1, Z_1)\\
			&\overset{\eqref{connLC1}, \eqref{connLC2}}{=}Y_1(f_2)X_2(f_1)g(X_1, Z_1)+X_1(X_2(f_1)) g(Y_1, Z_1)+X_2(f_1)X_1(f_1)g(Y_1, Z_1)\\
			&\phantom{\overset{\eqref{connLC1}, \eqref{connLC2}}{=}}+X_2(f_1)Y_1(f_1)g(X_1, Z_1)-X_2(f_1)Z_1(f_1)g(X_1, Y_1).
		\end{split}
	\end{equation*}
	Replacing this  formula and the one obtained from it by interchanging the roles of $X_1$ and $Y_1$ into \eqref{defriemcurv} we obtain \eqref{Riem1112}.
	We now show \eqref{Riem1111} by computing as follows:
\begin{equation*}
	\begin{split}
		&R^g(X_1, Y_1, Y_1, X_1)=	g(\nabla^g_{X_1}\nabla^g_{Y_1} Y_1, X_1)-	g(\nabla^g_{Y_1}\nabla^g_{X_1} Y_1, X_1)\\	
		&=X_1(g(\nabla^g_{Y_1}Y_1, X_1))-g(\nabla^g_{Y_1}Y_1, \nabla^g_{X_1}X_1)-Y_1(g(\nabla^g_{X_1}Y_1, X_1))+g(\nabla^g_{X_1}Y_1, \nabla^g_{Y_1} X_1)\\
		&\overset{\eqref{connLC1}}{=}X_1(g(\nabla^{g_1}_{Y_1}Y_1, X_1))+2X_1(Y_1(f_1)g(X_1,Y_1))-X_1(|Y_1|_{g}^2  X_1(f_1))\\
		&\phantom{xx}-g(2X_1(f_1)X_1-|X_1|^2_g\,\di f_1, 2Y_1(f_1)Y_1-|Y_1|^2_g\,\di f_1)\\
		&\phantom{xx}-Y_1(g(\nabla^{g_1}_{X_1}Y_1, X_1))-Y_1(Y_1(f_1)|X_1|^2_{g})\\
		&\phantom{xx}+(X_1(f_1))^2|Y_1|^2_g+(Y_1(f_1))^2|X_1|^2_g+(g(X_1, Y_1))^2|\di f_1|^2_g-2X_1(f_1)\cdot Y_1(f_1)\,g(X_1, Y_1)\\
		&=R^{g_1}(X_1, Y_1, Y_1, X_1)\\
		&\phantom{xx}+2X_1(Y_1(f_1))g(X_1,Y_1)+4X_1(f_1)\cdot Y_1(f_1)g(X_1, Y_1)- X_1(X_1(f_1))\, |Y_1|^2_g-2(X_1(f_1))^2|Y_1|^2_g\\
		&\phantom{xx}-4X_1(f_1)\cdot Y_1(f_1)g(X_1, Y_1)+2(X_1(f_1))^2|Y_1|^2_g+2(Y_1(f_1))^2|X_1|^2_g-|\di f_1|^2_g|X_1|^2_g|Y_1|^2_g\\
		&\phantom{xx}-Y_1(Y_1(f_1))|X_1|^2_g-2(Y_1(f_1))^2|X_1|^2_g\\
		&\phantom{xx}+(X_1(f_1))^2|Y_1|^2_g+(Y_1(f_1))^2|X_1|^2_g+(g(X_1, Y_1))^2|\di f_1|^2_g-2X_1(f_1)\cdot Y_1(f_1)\,g(X_1, Y_1),
	\end{split}
\end{equation*}			
which yields \eqref{Riem1111}. Similarly, we compute:
\begin{equation*}
\begin{split}
&R^g(X_1, Y_2, Y_2, X_1)=	g(\nabla^g_{X_1}\nabla^g_{Y_2} Y_2, X_1)-	g(\nabla^g_{Y_2}\nabla^g_{X_1} Y_2, X_1)\\	
&=X_1(g(\nabla^g_{Y_2}Y_2, X_1))-g(\nabla^g_{Y_2}Y_2, \nabla^g_{X_1}X_1)-Y_2(g(\nabla^g_{X_1}Y_2, X_1))+g(\nabla^g_{X_1}Y_2, \nabla^g_{Y_2} X_1)\\
&\overset{\eqref{connLC1}, \eqref{connLC1bis}}{=}-X_1(|Y_2|_g^2 X_1(f_2))+2|Y_2|_g^2\,X_1(f_1)\cdot X_1(f_2)+2|X_1|_g^2\,Y_2(f_1)\cdot Y_2(f_2)-|X_1|_g^2|Y_2|^2_g\, g(\di f_1, \di f_2)\\
&\phantom{XXx}-Y_2(|X_1|^2_gY_2(f_1))+(X_1(f_2))^2|Y_2|_g^2+(Y_2(f_1))^2|X_1|_g^2\\
&=- (X_1(f_2))^2\, |Y_2|_g^2-X_1(X_1(f_2))\, |Y_2|_g^2-(Y_2(f_1))^2\, |X_1|_g^2-Y_2(Y_2(f_1))\, |X_1|_g^2\\
&\phantom{xx}+2 X_1(f_1)\cdot X_1(f_2)|Y_2|_g^2+2Y_2(f_1)\cdot Y_2(f_2)\, |X_1|_g^2-g(\di f_1, \di f_2)\, |X_1|_g^2|Y_2|^2_g,
\end{split}
\end{equation*}	
which yields \eqref{Riem1221}.
\end{proof}

\begin{lemma}\label{lemmaricci}
	The Ricci curvature tensor $\mathrm{Ric}^g$ of a conformal product  metric $g=e^{2f_1}g_1+e^{2f_2}g_2$ on $M$ is given by the following formulas, for each vector fields $X_1\in\mathcal{C}^{\infty}(\mathrm{T}M_1)$ and $X_2\in\mathcal{C}^{\infty}(\mathrm{T}M_2)$:
	\begin{equation}\label{Ric12}
		\mathrm{Ric}^g(X_1,X_2)=(1-n_1)X_1(X_2(f_1))+(1-n_2)X_2(X_1(f_2))+(2-n)X_1(f_2)\cdot X_2(f_1),
	\end{equation}	
 \begin{equation}\label{Ric11}
	\begin{split}
		\mathrm{Ric}^g(X_1,X_1)&=\mathrm{Ric}^{g_1}(X_1,X_1)+ ( e^{-2f_2}\Delta_2f_1+e^{-2f_1}\Delta_1f_1)|X_1|_{g}^2+(2-n_2)g(\di_2 f_1, \di_2 f_2)|X_1|_g^2\\
		&\phantom{=}-[n_2\,g(\di_1f_1, \di_1 f_2)
		+n_1|\di_2f_1|_{g}^2-(2-n_1)|\di_1 f_1|_{g}^2]|X_1|_g^2\\
		&\phantom{=}+(2-n_1)[\mathrm{Hess}^{g_1}(f_1)(X_1, X_1)-(X_1(f_1))^2]\\
		&\phantom{=}-n_2[\mathrm{Hess}^{g_1}(f_2)(X_1, X_1)+(X_1(f_2))^2-2X_1(f_1)\cdot X_1(f_2)],
	\end{split}
 \end{equation}	
where for every function $f\in \mathcal{C}^\infty(M)$ and vector field $X_1\in \mathcal{C}^{\infty}(\mathrm{T}M_1)\subset \mathcal{C}^{\infty}(\mathrm{T}M)$, we denote by $\mathrm{Hess}^{g_1}(f)(X_1, X_1):=X_1(X_1(f))-(\nabla^{g_1}_{X_1}X_1)(f)$ the Hessian with respect to $g_1$ of the restriction of $f$ to the $M_1$-leaves of $M$.
\end{lemma}

\begin{proof}
Considering a local $g$-orthonormal basis of the form $\{e^{-f_1}\alpha_i,e^{-f_2} \beta_j\}_{1\leq i\leq n_1, 1\leq j\leq n_2}$, where $\{\alpha_i\}_{1\leq i\leq n_1}$ is a local $g_1$-orthonormal basis on $M_1$ and  $\{\beta_j\}_{1\leq j\leq n_2}$  is a local $g_2$-orthonormal basis on $M_2$, we write:
\begin{equation}\label{eqric}
	\mathrm{Ric}^g(X_1,X_2)=e^{-2f_1}\sum_{i=1}^{n_1}\mathrm{R}^g(X_1, \alpha_i, \alpha_i, X_2)+e^{-2f_2}\sum_{j=1}^{n_2}\mathrm{R}^g(X_1, \beta_j, \beta_j, X_2).
\end{equation}	
We compute separately the first term on the right-hand side of \eqref{eqric} using the formulas obtained for the Riemannian curvature tensor in Lemma~\ref{lemmariemtensor}:
\begin{equation*}
	\begin{split}
		e^{-2f_1}\sum_{i=1}^{n_1}\mathrm{R}^g(X_1, \alpha_i, \alpha_i, X_2)\overset{\eqref{Riem1112}}{=}&e^{-2f_1}\sum_{i=1}^{n_1}g(X_1, \alpha_i)[\alpha_i(X_2(f_1))-\alpha_i(f_2)\cdot X_2(f_1)]\\
		&-e^{-2f_1}\sum_{i=1}^{n_1}g(\alpha_i, \alpha_i)[X_1(X_2(f_1))-X_1(f_2)\cdot X_2(f_1)]\\
		=&X_1(X_2(f_1))-X_1(f_2)\cdot X_2(f_1)-n_1 X_1(X_2(f_1)) +n_1X_1(f_2)\cdot X_2(f_1)\\
		=&(1-n_1)[X_1(X_2(f_1))-X_1(f_2)\cdot X_2(f_1)].
	\end{split}
\end{equation*}	
The second term  in \eqref{eqric} is computed similarly:
\begin{equation*}
	\begin{split}
		e^{-2f_2}\sum_{j=1}^{n_2}\mathrm{R}^g(X_1, \beta_j, \beta_j, X_2)=&e^{-2f_2}\sum_{j=1}^{n_2}\mathrm{R}^g(X_2, \beta_j, \beta_j, X_1)\\
		\overset{\eqref{Riem2221}}{=}&e^{-2f_2}\sum_{j=1}^{n_2} g(X_2, \beta_j)[\beta_j(X_1(f_2))-\beta_j(f_1)\cdot X_1(f_2)]\\
		&-e^{-2f_2}\sum_{j=1}^{n_2} g(\beta_j,\beta_j)[X_2(X_1(f_2))-X_2(f_1)\cdot X_1(f_2)]\\
		=&X_2(X_1(f_2))-X_1(f_2)\cdot X_2(f_1)-n_2X_2(X_1(f_2))+n_2X_1(f_2)\cdot X_2(f_1)\\
		=&(1-n_2)[X_2(X_1(f_2))-X_1(f_2)\cdot X_2(f_1)].
	\end{split}
\end{equation*}	
Replacing these two formulas in \eqref{eqric} we obtain \eqref{Ric12}.

Considering the same local orthormal bases as above, namely a local $g_1$-orthonormal basis $\{\alpha_i\}_{1\leq i\leq n_1}$ on $M_1$ and a local $g_2$-orthonormal $\{\beta_j\}_{1\leq j\leq n_2}$ on $M_2$,  we write:
\begin{equation}\label{eqric11}
	\mathrm{Ric}^g(X_1,X_1)=e^{-2f_1}\sum_{i=1}^{n_1}\mathrm{R}^g(X_1, \alpha_i, \alpha_i, X_1)+e^{-2f_2}\sum_{j=1}^{n_2}\mathrm{R}^g(X_1, \beta_j, \beta_j, X_1),
\end{equation}	
where the first term on the right-hand side of \eqref{eqric11} is computed as follows:
\begin{equation*}
	\begin{split}
		&e^{-2f_1}\sum_{i=1}^{n_1}\mathrm{R}^g(X_1, \alpha_i, \alpha_i, X_1)\overset{\eqref{Riem1111}}{=}\mathrm{Ric}^{g_1}(X_1,X_1)\\
		&\phantom{xx}+e^{-2f_1}\sum_{i=1}^{n_1}[2X_1(\alpha_i(f_1))g(X_1,\alpha_i)-2X_1(f_1)\cdot \alpha_i(f_1)\,g(X_1, \alpha_i)]\\
		&\phantom{xx}- \sum_{i=1}^{n_1}[X_1(X_1(f_1))\, |\alpha_i|^2_g+\alpha_i(\alpha_i(f_1))\, |X_1|^2_g-(X_1(f_1))^2|\alpha_i|^2_g-(\alpha_i(f_1))^2|X_1|^2_g]\\
		&\phantom{xx}-\sum_{i=1}^{n_1}[|\di f_1|^2_g|X_1|^2_g|\alpha_i|^2_g
		-(g(X_1, \alpha_i))^2|\di f_1|^2_g]\\
		&=\mathrm{Ric}^{g_1}(X_1,X_1)+e^{-2f_1}\sum_{i=1}^{n_1}[2X_1(\alpha_i(f_1))g(X_1,\alpha_i)-2X_1(f_1)\cdot \alpha_i(f_1)\,g(X_1, \alpha_i)]\\
		&\phantom{xx}- e^{-2f_1}\sum_{i=1}^{n_1}[X_1(X_1(f_1))\, |\alpha_i|^2_g+\alpha_i(\alpha_i(f_1))\, |X_1|^2_g-(X_1(f_1))^2|\alpha_i|^2_g-(\alpha_i(f_1))^2|X_1|^2_g]\\
		&\phantom{xx}-e^{-2f_1}\sum_{i=1}^{n_1}[|\di f_1|^2_g|X_1|^2_g|\alpha_i|^2_g
		-(g(X_1, \alpha_i))^2|\di f_1|^2_g]\\
		&=	\mathrm{Ric}^{g_1}(X_1,X_1)+(n_1-2)(X_1(f_1))^2-(n_1-2)X_1(X_1(f_1))+e^{-2f_1}\Delta_1 f_1\cdot |X_1|^2_g\\
		&\phantom{xx}+(2-n_1)|\di_1 f_1|^2_{g}|X_1|^2_g
	+(1-n_1)|\di_2 f_1|^2_g|X_1|^2_g.
	 \end{split}
\end{equation*}	
Similarly, the second term on the right-hand side of  \eqref{eqric11} is computed as follows:
\begin{equation*}
	\begin{split}
		e^{-2f_2}\sum_{j=1}^{n_2}\mathrm{R}^g(X_1, \beta_j, \beta_j, X_1)\overset{\eqref{Riem1221}}{=}&-n_2[ (X_1(f_2))^2  +X_1(X_1(f_2))-2X_1(f_1)\cdot X_1(f_2)]\\
		&+[(2-n_2) g(\di_2 f_1, \di_2 f_2)-n_2g(\di_1 f_1, \di_1 f_2)]\,|X_1|_g^2\\
		&-[|\di_2 f_1|^2_g+e^{-2f_2}\Delta_2 f_1]\,|X_1|_g^2.
	\end{split}
\end{equation*}	
	Altogether, replacing the last two formulas in \eqref{eqric11}, we obtain \eqref{Ric11}, which finishes the proof of the lemma.
\end{proof}

By symmetry, permuting the indexes in \eqref{Ric11}, we obtain  the analogous formula:
\begin{equation}\label{Ric22}
	\begin{split}
		\mathrm{Ric}^g(X_2,X_2)&=\mathrm{Ric}^{g_2}(X_2,X_2)+ ( e^{-2f_1}\Delta_1f_2+e^{-2f_2}\Delta_2f_2)|X_2|_g^2+(2-n_1)g(\di_1 f_1, \di_1 f_2)|X_2|_g^2\\
		&\phantom{=}-[n_1\, g(\di_1f_2, \di_2 f_1)+n_2|\di_1f_2|_g^2-(2-n_2)|\di_2 f_2|_g^2]|X_2|_g^2\\
		&\phantom{=}+(2-n_2)[\mathrm{Hess}^{g_2}(f_2)(X_2, X_2)-(X_2(f_2))^2]\\
		&\phantom{=}-n_1[\mathrm{Hess}^{g_2}(f_1)(X_2, X_2)+(X_2(f_1))^2-2X_2(f_2)\cdot X_2(f_1)].
	\end{split}
\end{equation}	

	\section{Main Result}
	In this section we give the proof  of Theorem~\ref{mainthm}. Let us start with the following equivalent characterization of the geometric assumption in Theorem~\ref{mainthm} about the Lee form of the adapted Weyl connection:
	
	\begin{lemma}\label{equivlemma}
		Let $(M_1,c_1)$ and $(M_2,c_2)$ be two compact conformal manifolds and let $c$ be a conformal product structure on $M_1\times M_2$, with adapted Weyl connection $\nabla$. A metric $g$ in the conformal class $c$ has the property that the restriction to $\T M_2$ of the Lee form of $\nabla$ with respect to $g$ is $\nabla$-parallel in the direction of $\T M_1$ if and only if there exist functions $f_1, f_2\in\mathcal{C}^\infty(M)$ satisfying $\di_1f_1=0$ and metrics $g_i\in c_i$, such that $g=e^{2f_1}g_1+e^{2f_2}g_2$.
	\end{lemma}
	
	\begin{proof}
		Let $g$ be a metric in $c$. Then there  exist metrics $g_i$ on $M_i$ and functions $f_1$, $f_2$ on~$M$, such that $g=e^{2f_1}g_1+e^{2f_2}g_2$. Since the Lee form of the Weyl connection $\nabla$ with respect to~$g$ is given by $\theta^g=-\di_1 f_2-\di_2 f_1$, it follows that its restriction to $\T M_2$ is $\theta^g_2=-\di_2 f_1$. We compute for all vector fields $X_i\in\mathcal{C}^\infty(\T M_i)$:
		\begin{equation*}
			\begin{split}
				(\nabla_{X_1}\theta^g_2)(X_2)&=-(\nabla_{X_1}\di_2 f_1)(X_2)=-X_1(X_2(f_1))+\di_2 f_1(\nabla_{X_1} X_2)\overset{\eqref{nabla12}}{=}-\di_1\di_2 f_1(X_1, X_2).
			\end{split}	
		\end{equation*}
		This equation shows that $\theta^g_2$ satisfies $\nabla_{X_1}\theta^g_2=0$ for all $X_1\in\mathcal{C}^\infty(\T M_1)$ if and only if \mbox{$\di_1\di_2 f_1=0$}. By Lemma~\ref{lemmasum} applied to $f_1$, there exist functions $a_i\in\mathcal{C}^\infty(M_i)$, such that $f_1=a_1+a_2$. If we replace $f_1$ by $a_2$ and the metric $g_1$ by $e^{2a_1}g_1$, then we may assume that $\di_1 f_1=0$, which finishes the proof of the lemma.
	\end{proof}

    Let now $g$ be an Einstein metric on $M_1\times M_2$ satisfying the assumption of Theorem~\ref{mainthm}. According to Lemma~\ref{equivlemma}, the  metric $g$ can be written as $g=e^{2f_1}g_1+e^{2f_2}g_2$, where $f_1$ satisfies $\di_1f_1=0$. Under these assumptions, the formulas for the Ricci curvature of $g$ become much simpler. First, because the metric $g$ is Einstein, the left-hand side of \eqref{Ric12} vanishes. On the other hand,  $\di_1 f_1=0$ implies that the term $X_1(X_2(f_1))$ on the right-hand side of \eqref{Ric12} also vanishes. Thus, \eqref{Ric12} yields the following equality:
	\begin{equation}\label{Einstein1}
		(n-2)X_1(f_2)\cdot X_2(f_1)=(n_2-1)X_1(X_2(f_2)), 
	\end{equation}
for all vector fields $X_i\in\mathcal{C}^\infty(\mathrm{T}M_i)$.
	Let $\lambda\in\mathbb{R}$ be the Einstein constant of $g$, \emph{i.e.} $\Ric^g=\lambda g$. Under the assumption that $\di_1 f_1=0$, Equation \eqref{Ric11} reads:
	\begin{equation}\label{Ric11s}
		\begin{split}
			\lambda|X_1|^2_g&=\mathrm{Ric}^{g_1}(X_1,X_1)+|X_1|^2_g[ e^{-2f_2}\Delta_2f_1+(2-n_2)g(\di_2 f_1, \di_2 f_2)-n_1|\di_2 f_1|^2_g]\\
			&\phantom{=}-n_2(\mathrm{Hess}^{g_1}(f_2)(X_1, X_1)+X_1(f_2)^2), 
		\end{split}
	\end{equation}	
	for all vector fields $X_i\in\mathcal{C}^\infty(\mathrm{T}M_i)$.
	
	The key argument for the proof of  Theorem ~\ref{mainthm} is the following:
	\begin{lemma}\label{keylemma} The function $f_2$ satisfies $\di_1\di_2 f_2=0$.
	\end{lemma}
	
\begin{proof} We introduce the following closed subsets of $M$:
	$$C_1:=\{x\in M\, |\, (\di_1 f_2)_x=0\}\quad \text{and} \quad C_2:=\{x\in M\, |\, (\di_2 f_1)_x=0\}.$$

	The proof of Lemma \ref{keylemma} will be split into three cases, according to the possible dimensions of the factors $M_1$ and $M_2$.
		
		{\bf Case 1.} In this first case we assume that the dimension of each factor is at least $2$, \emph{i.e.} $n_1\geq 2$ and $n_2\geq 2$. Then, Equality \eqref{Einstein1}  reads
		\begin{equation}\label{x1x2}
			X_1(X_2(f_2))=\frac{n-2}{n_2-1} X_1(f_2)\cdot X_2(f_1),\qquad \forall X_i\in\mathcal{C}^\infty(\mathrm{T}M_i).
		\end{equation}
		By definition, $\di_2 f_1=0$ on $C_2$. Hence, Equation \eqref{x1x2} implies that $\di_1\di_2 f_2=0$ at all points of~$C_2$. If $M=C_2$ we are done, so we assume for the remaining part of the argument that the open set $M\setminus C_2$ is non-empty. Equation \eqref{Ric11s} can be equivalently written
		\begin{equation}\label{eqHess}
		\mathrm{Hess}^{g_1}(f_2)(X_1, X_1)+(X_1(f_2))^2=\frac{1}{n_2}\mathrm{Ric}^{g_1}(X_1,X_1)+\varphi |X_1|_{g_1}^2,
		\end{equation}
		where we denote by $\varphi$ the function $\displaystyle \varphi:=\frac{ 1}{n_2} e^{2f_1}(e^{-2f_2}\Delta_2f_1+(2-n_2)(\di_2 f_1, \di_2 f_2)-n_1|\di_2 f_1|^2_g-\lambda)$. Differentiating \eqref{eqHess} in the direction of $X_2$ and choosing $X_1$ to be $\nabla^{g_1}$-parallel at the point where the computation is done, we obtain:
		\begin{equation*}
			\begin{split}
				|X_1|_{g_1}^2 X_2(\varphi )&=2X_2(X_1(f_2))\cdot X_1(f_2)+\mathrm{Hess}^{g_1}(X_2(f_2))(X_1, X_1)\\
				&\overset{\eqref{x1x2}}{=} \frac{2(n-2)}{n_2-1}(X_1(f_2))^2\cdot X_2(f_1)+X_1\left(\frac{n-2}{n_2-1}X_1(f_2)\cdot X_2(f_1) \right)\\
				&=\frac{n-2}{n_2-1}\left[2(X_1(f_2))^2 +\mathrm{Hess}^{g_1}(f_2)(X_1, X_1) \right]\cdot X_2(f_1).
			\end{split}	
		\end{equation*}	
		By tensoriality, this equation holds for all vector fields $X_i\in\mathcal{C}^\infty(\mathrm{T}M_i)$.
Let $x\in M\setminus C_2$ be an arbitrary point. By definition, there exists $X_2\in\mathcal{C}^\infty(\mathrm{T}M_2)$ such that $X_2(f_1)\ne 0$ on some neighborhood $V_x$ of $x$. 
		Restricting the above equality to $V_x$, we can write
		$$2(X_1(f_2))^2 +\mathrm{Hess}^{g_1}(f_2)(X_1, X_1) =|X_1|_{g_1}^2 \varphi_1,$$
		where $\displaystyle \varphi_1:=\frac{n_2-1}{n-2}\cdot \frac{X_2(\varphi )}{X_2(f_1)}$, which is well-defined on $V_x$. Differentiating this last equation again in the direction of $X_2\in \mathcal{C}^\infty(\mathrm{T} M_2)$, a similar computation using \eqref{x1x2} shows that
		$$4(X_1(f_2))^2 +\mathrm{Hess}^{g_1}(f_2)(X_1, X_1) =|X_1|_{g_1}^2 \varphi_2,$$
		where $\displaystyle\varphi_2:=\frac{n_2-1}{n-2}\cdot \frac{X_2(\varphi_1)}{X_2(f_1)}$. The difference of the last two identities reads
		$$(X_1(f_2))^2 = |X_1|_{g_1}^2 \varphi_3,$$
		where $\varphi_3:=\frac{1}{2}(\varphi_2-\varphi_1)$ on $V_x$. Since this equality holds for all vector fields $X_1\in\mathcal{C}^{\infty}(\mathrm{T}M_1)$, we obtain that $\di_1 f_2\otimes \di_1 f_2=\varphi_3 g_1$ on $V_x$. Since $n_1\geq 2$ and $V_x$ is non-empty, it follows that $\varphi_3=0$ and thus $\di_1 f_2=0$ on the open set $V_x$, thus on the whole $M\setminus C_2$ since $x$ was arbitrary. By \eqref{Einstein1} it follows that $\di_1\di_2 f_2=0$ on $M\setminus C_2$. But we have already noticed that $\di_1\di_2 f_2=0$ on $C_2$, so finally $\di_1\di_2 f_2=0$ on $M$.
		
		\bigskip
		
		{\bf Case 2.} Assume that $n_2=1$ and $n_1>1$. Since $X_1(f_1)=0$ for all $X_1\in \mathcal{C}^\infty(\mathrm{T} M_1)$, Equality \eqref{Ric12} yields that
		$$X_1(f_2)\cdot X_2(f_1)=0, $$
		for all vector fields $X_i\in\mathcal{C}^\infty(\mathrm{T}M_i)$ which are $\nabla^{g_i}$-parallel at the point where the computation is done.
		Thus in this case we have $M=C_1\cup C_2$, which in particular implies that the union of the interiors $\overset{\circ}{C_1}\cup  \overset{\circ}{C_2}$ is a dense subset in $M$. 
		
		We claim that  $\di_2\di_1 f_2=0$ at each point of $\overset{\circ}{C_1}\cup  \overset{\circ}{C_2}$. At points of $\overset{\circ}{C_1}$, this follows directly by differentiating the relation $\di_1f_2=0$ which holds on the open set $\overset{\circ}{C_1}$. 
		
		On the other hand, since $\di_2f_1|_{ \overset{\circ}{C_2}}=0$ and $\di_1f_1=0$ by assumption, it follows that $f_1$ is locally constant on $ \overset{\circ}{C_2}$. Thus, on $ \overset{\circ}{C_2}$, Equation \eqref{Ric11s} simplifies to:
		\begin{equation}
			\begin{split}
				\lambda|X_1|^2_g&=\mathrm{Ric}^{g_1}(X_1,X_1)-(\mathrm{Hess}^{g_1}(f_2)(X_1, X_1)+X_1(f_2)^2).
			\end{split}
		\end{equation}	
		
		Substituting $X_2=\frac{\partial}{\partial t}$ in \eqref{Ric22}, where $t$ denotes the arc length coordinate on $(M_2,g_2)$ and denoting by $f', f''$ the derivatives of a function $f$ with respect to $t$, we obtain on $ \overset{\circ}{C_2}$:
		$$\lambda e^{2f_2}=e^{2f_2}[e^{-2f_1}\Delta_1f_2-e^{-2f_2}f''_2-|\di_1f_2|^2_g+e^{-2f_2} (f'_2)^2]+f_2''-(f'_2)^2,$$
		or, equivalently:
		\begin{equation}\label{lambda}
			\lambda=e^{-2f_1}\Delta_1f_2-|\di_1f_2|^2_g=e^{-2f_1}(\Delta_1f_2-|\di_1f_2|_{g_1}^2).
		\end{equation}
		We now show that $\di_1\di_2 f_2=0$ on $\overset{\circ}{C_2}$, by considering the following subcases:
		\begin{enumerate}
			\item[a)] If $\overset{\circ}{C_1}=\emptyset$, then $\overset{\circ}{C_2}$ is dense in $M$, so \eqref{lambda} is satisfied on $M$. Evaluating \eqref{lambda} at a point of $M$ where $f_2$ attains its global maximum yields $\lambda \ge 0$, and at a point where $f_2$ attains its global minimum we obtain $\lambda\le 0$. Consequently $\lambda=0$, whence $\Delta_1f_2=|\di_1f_2|_{g_1}^2$ on $M$. Integrating on each slice $M_1\times \{x_2\}$ yields $\di_1f_2=0$ on $M$. In particular, also $\di_1\di_2 f_2=0$.
			\item[b)] If $\overset{\circ}{C_2}=\emptyset$, then $\overset{\circ}{C_1}$ is dense in $M$, so $\di_1 f_2=0$ on $M$. In particular, also in this case it follows that $\di_1\di_2 f_2=0$. 
			\item[c)] If $\overset{\circ}{C_1}\neq \emptyset$ and $\overset{\circ}{C_2}\neq \emptyset$, then we evaluate Equation \eqref{lambda} at a point from the intersection of the closures of $\overset{\circ}{C_1}$ and $\overset{\circ}{C_2}$, which is not empty (because the union of these closures is $M$ which is connected), yields $\lambda=0$. Hence, \eqref{lambda} further implies that $\Delta_1f_2=|\di_1f_2|_{g_1}^2$ on $\overset{\circ}{C_2}$.  Since this equality holds by definition on $\overset{\circ}{C_1}$, it then follows by density that $\Delta_1f_2=|\di_1f_2|_{g_1}^2$ on $M$. We conclude like in case a) that $\di_1\di_2 f_2=0$ on $M$, by integration on the slices $M_1\times \{x_2\}$.
		\end{enumerate}

		{\bf Case 3.} Assume that $n_1=1$ and $n_2> 1$.  Equation \eqref{Einstein1} reads:
		\begin{equation}\label{deriv22f2}
			\di_1\di_2f_2=\di_1f_2\wedge \di_2f_1.
		\end{equation}
		Hence, $\di_1\di_2 f_2=0$ on $C_1\cup C_2$.
		We denote by $U$ the open set $U:=M\setminus(C_1\cup C_2)$. If $U=\emptyset$, then the equality $\di_1\di_2 f_2=0$ holds on $M$ and we are done.
		Assume for the rest of the proof that $U\neq\emptyset$. Denoting by  $t$ the arc length coordinate on $(M_1,g_1)$, we remark that, by definition, $f'_2$ does not vanish at any point of $U$. Equation \eqref{Ric11} implies that 
		$$ \lambda e^{2f_1}=e^{2f_1}[e^{-2f_2}\Delta_2 f_1+(2-n_2)g(\di_2 f_2, \di_2 f_1)-|\di_2f_1|^2_g]-n_2(f''_2+(f'_2)^2),$$
		or, equivalently,
		\begin{equation}\label{Ric11sp}
			\lambda e^{2f_2}=\Delta_2 f_1+(2-n_2)g_2(\di_2 f_2, \di_2 f_1)-|\di_2f_1|_{g_2}^2-n_2e^{2f_2-2f_1}(f''_2+(f'_2)^2).
		\end{equation}
		Differentiating this equality with respect to $t$ yields:
		$$2 \lambda f'_2 e^{2f_2}=(2-n_2)g_2(\di_2 f'_2, \di_2 f_1)-n_2e^{2f_2-2f_1}[2f'_2(f''_2+(f'_2)^2)+f^{(3)}_2+2f'_2f''_2].$$
		From \eqref{deriv22f2} and $\di_1 f_1=0$, it follows that $\di_2 f'_2=f'_2\di_2 f_1$, so we have $$g_2(\di_2 f'_2, \di_2 f_1)=g_2(f'_2\di_2f_1, \di_2f_1)=f_2'|\di_2f_1|^2_{g_2},$$
		which replaced in the above equality yields
		\begin{equation}\label{f2deriv}
			2 \lambda f'_2 e^{2f_2}=(2-n_2)f_2'|\di_2f_1|_{g_2}^2-n_2e^{2f_2-2f_1}[4f'_2f''_2+2(f'_2)^3+f^{(3)}_2].
		\end{equation}
		Dividing this equality by $f'_2$, which does not vanish on $U$, and then differentiating again with respect to $t$, we obtain:
		\begin{equation*}
			4 \lambda f'_2 e^{2f_2}=-n_2e^{2f_2-2f_1}[12f'_2f''_2+4(f'_2)^3+6f^{(3)}_2+f_2^{(4)}(f_2')^{-1}- (f_2')^{-2}f_2''f_2^{(3)}],
		\end{equation*}
		or, equivalently, after dividing by $-n_2f'_2e^{2f_2-2f_1}$:
		\begin{equation}\label{eqlambda}
			-\frac{4 \lambda e^{2f_1}}{n_2}=12f''_2+4(f'_2)^2+\frac{6f^{(3)}_2}{f'_2}+\frac{f_2^{(4)}}{(f'_2)^2}- \frac{f_2''f_2^{(3)}}{(f'_2)^3}.
		\end{equation}
		Differentiating this equality once more with respect to $t$ yields:
		\begin{equation*}
			0=\underbrace{12f^{(3)}_2}_{=:A_1}+\underbrace{8f'_2f_2''}_{=:A_2}+\underbrace{\frac{6f^{(4)}_2}{f'_2}-\frac{6f''_2f^{(3)}_2}{(f'_2)^2}}_{=:A_0}+\underbrace{\frac{f_2^{(5)}}{(f'_2)^2}-\frac{2f''_2f_2^{(4)}}{(f'_2)^3}- \frac{(f_2^{(3)})^2+f''_2f_2^{(4)}}{(f'_2)^3}+\frac{3(f''_2)^2f_2^{(3)}}{(f'_2)^4}}_{=:A_{-1}}.
		\end{equation*}
		
		We have introduced the above notation  $A_\ell$, for $\ell\in\{-1,0,1,2\}$ motivated by the fact that each $A_\ell$ is a rational fraction of degree $\ell$ in the derivatives $f_2^{(k)}$ of $f_2$ with respect to $t$. Notice that Equality \eqref{deriv22f2} together with the fact that $\di_1f_1=0$ yields
		\begin{equation}\label{kderiv22f2}
			\di_2f^{(k)}_2=f^{(k)}_2\di_2f_1, \quad \text{ for all } k\geq 1.
		\end{equation}
		 The following general lemma will be applied to the above defined functions $A_\ell$, for $\ell\in\{-1,0,1,2\}$.
		\begin{lemma}\label{homogen}
			Let $f_1$ and $f_2$ be two functions on $M$, satisfying \eqref{kderiv22f2}. 
			\begin{enumerate}
				\item[a)] For any homogeneous polynomial $P$ of $s\ge 1$ variables and of degree $p\ge 0$, and for every positive integers $k_1,\ldots,k_s$, the following relation holds:
				$$\di_2 P(f_2^{(k_1)}, \cdots ,f_2^{(k_s)})=p\cdot P(f_2^{(k_1)}, \cdots ,f_2^{(k_s)})\cdot \di_2 f_1.$$
				\item[b)] For any two homogeneous polynomials $P$ and $Q$ of $s\ge 1$ variables and of degree $p$, respectively $q$, and for every positive integers $k_1,\ldots,k_s$, the following relation holds:
				$$\di_2 \left(\frac{P}{Q}(f_2^{(k_1)}, \cdots ,f_2^{(k_s)})\right)=(p-q)\cdot \frac{P}{Q}(f_2^{(k_1)}, \cdots ,f_2^{(k_s)})\cdot \di_2 f_1.$$
			\end{enumerate}	
		\end{lemma}	
		\begin{proof} a) Follows directly by \eqref{kderiv22f2} applied to each monomial of $P$. 
		
		b) Follows from a) applied to $P$ and $Q$.
		\end{proof}
		
		Applying Lemma \ref{homogen} to $A_\ell$ yields $\di_2A_\ell=\ell A_\ell\cdot \di_2 f_1$, for all $\ell\in\{-1,0,1,2\}$. Thus, applying $\di_2$ to the equality $A_1+A_2+A_0+A_{-1}=0$ implies that $A_1+2A_2-A_{-1}=0$, since $\di_2 f_1$ does not vanish on $U$. Applying again  $\di_2$ to this relation yields $A_1+4A_2+A_{-1}=0$. Repeating the same argument once again yields $A_1+8A_2-A_{-1}=0$. This last equality together with the initial equality $A_1+A_2+A_0+A_{-1}=0$ implies that $A_2=0$, \emph{i.e.} $f'_2f''_2=0$, which means that $f''_2=0$ on $U$, because $f'_2$ does not vanish on $U$. Replacing $f''_2=0$ in \eqref{eqlambda} we obtain the following equality on $U$:
		\begin{equation}\label{f12}
			(f'_2)^2=-\frac{\lambda e^{2f_1}}{n_2},
		\end{equation}
		which implies in particular that $\lambda<0$. Replacing now $f''_2=0$ in \eqref{f2deriv} yields the following equality on $U$:
		\begin{equation*}
			2 \lambda e^{2f_2}=(2-n_2)|\di_2f_1|_{g_2}^2-2n_2e^{2f_2-2f_1}(f'_2)^2.
		\end{equation*} The last two equalities imply that
		\begin{equation*}
			2 \lambda e^{2f_2}=(2-n_2)|\di_2f_1|_{g_2}^2+2n_2e^{2f_2-2f_1}\cdot \frac{\lambda e^{2f_1}}{n_2},
		\end{equation*}
		so $(2-n_2)|\di_2f_1|_{g_2}^2=0$. Since $\di_2f_1\neq 0$ on $U$, it follows that $n_2=2$. Replacing $f''_2=0$ and $n_2=2$ in \eqref{Ric11sp} yields 
		$$\lambda e^{2f_2}=\Delta_2 f_1-|\di_2f_1|_{g_2}^2-n_2e^{2f_2-2f_1}(f'_2)^2=\Delta_2 f_1-|\di_2f_1|_{g_2}^2+\lambda e^{2f_2}$$
		and thus $\Delta_2 f_1=|\di_2f_1|_{g_2}^2$ on $U$.
		
		This shows that the function $\varphi:=e^{-2f_2}(\Delta_2 f_1-|\di_2f_1|_{g_2}^2)$ vanishes on $U$ and $\overset{\circ}{C_2}$. Furthermore, using \eqref{Ric11s} we deduce that $\varphi=\lambda$ on $\overset{\circ}{C_1}$. In particular $\di\varphi=0$ on $\overset{\circ}{C_1}\cup \overset{\circ}{C_2}\cup U$ which is dense in $M$, whence $\varphi$ is constant. Moreover $U\ne\emptyset$ by assumption, so $\varphi=0$ on $M$. This shows that the previous relation  $\Delta_2 f_1=|\di_2f_1|_{g_2}^2$ actually holds on $M$. Like before, integrating over the leafs $\{x_1\}\times M_2$ yields $\di_2f_1=0$ on $M$, so $M=C_2$. Thus $U=\emptyset$, contradicting our assumption. This concludes the proof of Case 3, and thus of Lemma \ref{keylemma}.
		\end{proof}
We can now finish the proof of Theorem \ref{mainthm}. By Lemmas \ref{equivlemma} and  \ref{keylemma}, the Einstein metric $g$ representing the conformal product structure on $M_1\times M_2$ can be written $g=e^{2f_1}g_1+e^{2f_2}g_2$ where $g_1,g_2$ are Riemannian metrics on $M_1,M_2$ respectively, and $f_1,f_2$ are functions on $M_1\times M_2$ satisfying $\di_1 f_1=0$ and $\di_1\di_2 f_2=0$. The first equation shows that in fact $f_1\in C^\infty(M_2)$, whereas the second equation together with Lemma \ref{lemmasum} shows that $f_2=a_1+a_2$ for some functions $a_1\in C^\infty(M_1)$ and $a_2\in C^\infty(M_2)$. Thus the conformal class $c$ can be written as
$$c=[g]=[e^{2f_1}g_1+e^{2f_2}g_2]=[e^{2f_1}g_1+e^{2a_1+2a_2}g_2]=[e^{-2a_1}g_1+e^{-2f_1+2a_2}g_2]=[h_1+h_2],$$
where $h_1:=e^{-2a_1}g_1$ is a metric on $M_1$ and $h_2:=e^{-2f_1+2a_2}g_2$ is a metric on $M_2$. This concludes the proof of Theorem \ref{mainthm}.
	
As already mentioned in the introduction, the fact that the Einstein metric $g$ on $M_1\times M_2$ is conformal to the product metric $h_1+h_2$, implies by \cite[Thm. 3.2 and Cor. 3.4]{kr2016} that the conformal factor between $g$ and $h_1+h_2$ is a function which only depends on $M_1$ or on $M_2$, {\em i.e.}  $g$ is a warped product metric. However, the complete classification of warped product Einstein metrics on compact manifolds is not yet available \cite{hpw2012}, except when the base is 1-dimensional, cf. \cite{k1988,kr1997,kr2009}.

\end{document}